\documentclass[12pt]{amsart}

\usepackage{amssymb}
\usepackage{verbatim}
\usepackage[toc,page]{appendix}
\usepackage{mathrsfs}

\newtheorem{thm}{Theorem}[section]

\newtheorem{lem}[thm]{Lemma}
\newtheorem{cor}[thm]{Corollary}

\theoremstyle{definition}

\theoremstyle{remark}

\newtheorem*{rem}{Remark}

\numberwithin{equation}{section}

\newcommand{\R}{\mathbf{R}}  % The real numbers.
\newcommand{\N}{\mathbf{N}}
\newcommand{\C}{\mathbf{C}}

\newcommand{\HH}{\mathbf{H}}
\newcommand{\Mod}[1]{\ (\textup{mod}\ #1)}

\providecommand{\sgn}{\operatorname{sgn}}

\providecommand{\sgn}{\operatorname{sgn}}
\providecommand{\sym}{\operatorname{sym}}
\DeclareMathOperator{\res}{res}

\def \hyp {{}_2F_1}

\begin{document}

\title[Mean value result for a product of $L$-functions]{A mean value result
for a product of $GL(2)$ and $GL(3)$ $L$-functions}

\author[O. Balkanova]{Olga  Balkanova}
\address{University of Turku, Department of Mathematics and Statistics,
Turku, 20014, Finland}
\email{olgabalkanova@gmail.com}
\thanks{Research of Olga Balkanova is supported by Academy of Finland project
no. $293876$.}

\author[G. Bhowmik]{Gautami Bhowmik}
\address{Laboratoire Painlev\'{e} LABEX-CEMPI, Universit\'{e} Lille 1, 59655
Villeneuve d’Ascq Cedex, France}
\email{bhowmik@math.univ-lille1.fr}

\author[D. Frolenkov]{Dmitry  Frolenkov}
\address{National Research University Higher School of Economics, Moscow,
Russia and Steklov Mathematical Institute of Russian Academy of
Sciences, 8 Gubkina st., Moscow, 119991, Russia}
\thanks{Research of Dmitry Frolenkov is supported by the Russian Science Foundation under grant [14-11-00335] and performed in Khabarovsk Division of the Institute for Applied Mathematics, Far Eastern Branch, Russian Academy of Sciences}
\email{frolenkov@mi.ras.ru}

\author[N. Raulf]{Nicole Raulf}
\address{Laboratoire Painlev\'{e} LABEX-CEMPI, Universit\'{e} Lille 1,
59655 Villeneuve d’Ascq Cedex, France}
\email{nicole.raulf@math.univ-lille1.fr}

\begin{abstract}
In this paper various analytic techniques are combined in order 
to study the average of a product of a Hecke $L$-function and 
a symmetric square $L$-function at the central point in the weight 
aspect. The evaluation of the second main term relies on the theory 
of Maa{\ss} forms of half-integral weight and the Rankin-Selberg 
method. The error terms are bounded using the Liouville-Green 
approximation. 
\end{abstract}

\keywords{symmetric square L-functions; Eisenstein-Maa{\ss} series;
  Maa{\ss} forms of half-integral weight; Rankin-Selberg method;
  Liouville-Green method; WKB approximation.}
\subjclass[2010]{Primary: 11F12; Secondary: 33C05, 34E05, 34E20.}

\maketitle

\tableofcontents

%%%%%%%%%%%%%%%%%%%%%%%%%%%%%%%%%%%%%%%%%%%%%%%%%%%%%%%%%%%%%%%%

%%%%%%%%%%%%%%%%%%%%%%%%%%%%%%%%%%%%%%%%%%%%%%%%%%%%%%%%%%%%%%%%%%%%%%
\section{Introduction}
%%%%%%%%%%%%%%%%%%%%%%%%%%%%%%%%%%%%%%%%%%%%%%%%%%%%%%%%%%%%%%%%%%%%%%

The asymptotic evaluation of moments of $L$-functions is not only 
an important tool to solve problems in number theory and arithmetic 
geometry, but is also a subject of independent interest. Various 
conjectures predict the shape of the main terms for all moments of 
$L$-functions within a family of certain type of symmetry (see e.g.\ 
\cite{CFKRS, DGH}). Even though exact results for moments are till 
now known only for small values this is already sufficient for many 
applications. 
See, for example, \cite{IS}.
In this regard, the quality of the asymptotic error estimates plays 
a crucial role.

In general, there are three main techniques used for determining the 
moments: we can apply the approximate functional equation, the 
Rankin-Selberg method or the method of analytic continuation, each of 
which has certain advantages. The method of the approximate functional 
equation which allows the bypassing of convergence problems
is the most common approach. The method of analytic continuation and 
the Rankin-Selberg method however reveal the structure of the mean 
values and yield exact formulas for the moments. Consequently, the 
theory of special functions can be used in order to prove sharp error 
estimates. 

The problem we consider here requires the combination of all three 
methods. We study the asymptotic 
behaviour of the first moment of the product of the Hecke 
$ L $-function $ L(f, 1/2) $ and the symmetric square $L$-function 
$ L(\sym^2f, 1/2) $ in the weight aspect on average, where $ f $ 
runs over the space $ H_{4k} $ of primitive forms of weight $ 4k $, 
$ k \in \N $.

The most interesting phenomenon of this moment is the presence of the 
non-diagonal main term, which is smaller in size and depends on the 
special value of the double Dirichlet series
\begin{equation} 
L^{-}_{g}(s) = 
\frac{1}{8} \sum_{n<0} \frac{\mathscr{L}_{4n}(1/2)}{|n|^{s+1/2}},
\end{equation}
\begin{equation}\label{dirserL}
\mathscr{L}_n(s) = \frac{\zeta(2s)}{\zeta(s)} \sum_{q=1}^{\infty}
\frac{1}{q^s} \left(\sum_{1\leq t \leq 2q; \, t^2 \equiv n \Mod{4q}}1\right).
\end{equation}

Isolating this non-diagonal main term and applying the Liouville-Green 
method for error estimates, we prove an asymptotic formula with an 
arbitrary power saving error term.

Before stating the main theorem 
we introduce some notation. Let $ \langle f, f\rangle_1 $ be the 
Petersson inner product on the space of level $1$ holomorphic 
modular forms and the standard harmonic weight we denote by 
\begin{equation}
\omega(f) := \frac{\Gamma(2k-1)}{(4\pi)^{2k-1} \langle f, f\rangle_1}. 
\end{equation} 
Then we have the following asymptotic formula. 

\begin{thm} \label{mainthm} 
Let $h \in C_{0}^{\infty}(\R^{+})$ be a non-negative, compactly supported 
function on the interval $ [\theta_1, \theta_2] $, $ \theta_2 > \theta_1 
> 0 $, and
\begin{equation}
\|h^{(n)}\|_1\ll 1 \text{ for all }n \geq 0.
\end{equation}
Then, for any fixed $ A>0 $, we have
\begin{equation} \label{eq:mainformula}
\begin{split} 
\sum_{k} & h\left(\frac{4k}{K}\right) \sum_{f\in H_{4k}} \omega(f) L(f,1/2) 
L(\sym^2f, 1/2) = \\ 
& 
\frac{HK}{4} \zeta(3/2) 
\Biggl(2\log{K} - 3\log{2\pi} - 2 \log{2} + \frac{\pi}{2} + 3 \gamma 
+ \frac{2 \zeta'(3/2)}{\zeta(3/2)}+ \frac{2 H_1}{H} \Biggr) \\ 
& 
+ P_h(1/K) + L^{-}_{g}(1/4) \sqrt{K} Q_{h}(1/K) + O\left(K^{-A}\right),
\end{split} 
\end{equation}
where $ \zeta(s) $ is the Riemann zeta function and $ \gamma $ denotes 
the Euler constant. Furthermore, 
\begin{equation*}
H := \int_{0}^{\infty}h(y)dy, \quad H_1 := \int_{0}^{\infty}h(y)\log{y}dy
\end{equation*}
and $P_h(x)$, $Q_h(x)$ are polynomials in $x$ of degree $A-1$ and $A$, 
respectively, with coefficients depending on the function $h$.
\end{thm}

This is an example of mixed moments that were previously investigated, for example, in 
\cite{LRW, MS,NMH,RW}.

The proof of Theorem \ref{mainthm} consists of several steps. We start 
by combining the exact formula for the twisted first moment of symmetric 
square $ L $-functions and the approximate functional equation for the Hecke 
$ L $-function. 
Consequently, the average
\begin{equation*} 
\sum_{f\in H_{4k}} \omega(f) L(f,1/2) L(\sym^2f, 1/2) 
\end{equation*} 
splits into a diagonal main term, a non-diagonal main term plus a smaller 
contribution expressed in terms of the special functions
\begin{equation}
\Psi_k(x) :=
x^k \frac{\Gamma(k-1/4) \Gamma(k+1/4)}{\Gamma(2k)}
{}_2F_{1}\left(k-\frac{1}{4}, k+\frac{1}{4}; 2k; x\right),
\end{equation}
\begin{equation}\label{defphi}
\Phi_k(x) :=
\frac{\Gamma(k-1/4) \Gamma(3/4-k)}{\Gamma(1/2)}
{}_2F_{1}\left(k-\frac{1}{4}, \frac{3}{4}-k; 1/2; x\right),
\end{equation}
where $ \Gamma(s) $ is the Gamma function and $ {}_2F_{1}(a,b;c;x) $ is 
the Gauss hypergeometric function.

The diagonal main term is evaluated in Corollary \ref{cor:diagonal}. The 
Rankin-Selberg method serves to isolate the non-diagonal main term, as 
shown in Corollary \ref{cor:nondiagonal}.
Lemma \ref{lem:firsttype} provides an estimate for the error term of the 
first type.
Using the Liouville-Green method, we approximate the error term of the 
second type by the series in $Y_0$ and $J_0$ Bessel functions, average 
the result over $k$ and, as a consequence, prove estimate \eqref{eq:secondtype}.

To sum up, we obtain the full asymptotic expansion with an arbitrary power 
saving error term. This improves the results of \cite{BRS}, where an asymptotic 
formula for
\begin{equation} \label{eq:hybrid} 
\sum_{k} h\left(\frac{4k}{K}\right) \sum_{f\in H_{4k}} \omega(f) L(f,1/2) 
L(\sym^2f, 1/2)
\end{equation} 
was proved with a saving in the logarithmic term of the error using 
standard methods based on the approximate functional equation and 
\cite[Lemma~5.8]{Iw}.

We note that a similar approach allowed Khan \cite{K2} to evaluate 
the second moment
\begin{equation}\label{eq:sym2moment} 
\sum_{k} h\left(\frac{2k}{K}\right) \sum_{f\in H_{2k}} \omega(f) 
L^2(\sym^2f, 1/2)
\end{equation}
with a good power saving error term, namely $O(K^{\epsilon})$ for any 
$ \epsilon > 0 $. This fact can be easily explained by the exact formula 
(see \cite[Theorem~2.1]{BF1}) for the twisted first moment of symmetric 
square $L$-functions associated to $H_{2k}$. The off-diagonal term in this 
formula is 
\begin{equation}\label{eq:ndod}
(-1)^k \frac{\sqrt{2\pi}}{2l^{1/2}} \frac{\Gamma(k-1/4)}{\Gamma(k+1/4)} 
\mathscr{L}_{-4l^2}(1/2).
\end{equation}
To obtain asymptotics for \eqref{eq:sym2moment}, it is required to average 
\eqref{eq:ndod} over all values of $k$. Consequently, the contribution of 
this term is rather small due to the oscillating multiple $(-1)^k$. In the 
case of \eqref{eq:hybrid}, the summation is taken only over even values of 
$k$ because $L(f,1/2)$ is identically zero for $f \in H_{2k}$ with odd $k$. 
As a result, the oscillating multiple $(-1)^k$ disappears and the summand 
corresponding to \eqref{eq:ndod} becomes the second main term.

Note that the absolute value estimate for the non-diagonal main term in 
\eqref{eq:hybrid} is $K^{3/4+\theta}$, where $\theta$ denotes the subconvexity 
exponent for the Dirichlet series \eqref{dirserL}.
Taking the convexity bound $\theta=1/4$, we recover the asymptotic result 
of \cite{BRS}.
Using the best subconvexity result due to Conrey and Iwaniec \cite{CI} gives 
a power saving bound $K^{11/12+\epsilon}$. The Rankin-Selberg method allows 
removing the dependence on $\theta$ and improving the estimate to $K^{3/4}$. 
We finally observe that there are some cancellations between the diagonal 
and the non-diagonal main terms. This can be shown by direct computations 
but the most convenient way is to choose a suitable form of the approximate 
functional equation for $ L(f, s) $. Accordingly, we obtain the second main 
term of size $K^{1/2}$.

In the level aspect similar results for the Hecke congruence subgroup of 
prime level $q$ were obtained by Munshi and Sengupta \cite{MS} with the 
error estimate $O(q^{-1/8})$. The authors of \cite{MS} isolate only the 
diagonal main term.  Our expectation is that there is a second main term 
of size $q^{-1/2}$.

We remark that the additional average over $ k $ in 
\eqref{eq:mainformula} is only required to estimate the error term 
of the second type, which involves the highly oscillatory special function 
$ \Phi_k(x) $. It might be possible to smooth out the oscillations of 
$ \Phi_k(x) $ by using instead the average over $ l $ in \eqref{eq:ET2}, 
and, hence, prove an asymptotic formula for
\begin{equation} \label{hybridmoment}
\sum_{f\in H_{4k}}\omega(f)L(f,1/2)L(\sym^2f, 1/2).
\end{equation}
A consequence of this result is the simultaneous non-vanishing of the 
corresponding $L$-functions.

Finally, the methods of the present work shed some light on the structure of the second moment of symmetric square $L$-functions. In the level aspect the upper bound 
\begin{equation}
\sum_{f\in H_{2k}(N)}L^2(\sym^2f,1/2)\ll N^{1+\epsilon}\text{ for any } \epsilon>0\text{ and fixed }k,
\end{equation}
 was proved by Iwaniec and Michel \cite{IM} in 2001.  In the weight aspect this is one of the most challenging unsolved problems in the theory of $L$-functions. It is believed that finding asymptotic formula or even an upper bound for
\begin{equation}\label{eq:sym2secondmoment}\sum_{f\in H_{2k}}\omega(f)L^2(sym^2f, 1/2) \text{ as }
k\rightarrow \infty,
\end{equation}
is  out of reach by standard tools. 
Our results suggest the following approach. Combining exact formula \eqref{eq:product} and the approximate functional equation for $L(sym^2f, 1/2)$, one can isolate the diagonal and the off-diagonal main terms in \eqref{eq:sym2secondmoment}. The main difference with 
\eqref{hybridmoment} is the presence of the off-off-diagonal main terms in sums \eqref{eq:ET3} and \eqref{eq:ET2}. Technically, this can be explained as follows. 
The approximate functional equation for symmetric square $L$-functions is longer and, therefore, the required range includes a transition point at which special functions $\Phi_k(x)$ and $\Psi_k(x)$ change their behaviour. This  produces the third off-off-diagonal main term and  creates additional technical difficulties in estimating error terms.

\section{Main tools}

%%%%%%%%%%%%%%%%%%%%%%%%%%%%%%%%%%%%%%%%%%%%%%%%%%%%%%%%%%%%%%%%%%%%%%
\subsection{Theorem of M\"{u}ller}\label{muller}
%%%%%%%%%%%%%%%%%%%%%%%%%%%%%%%%%%%%%%%%%%%%%%%%%%%%%%%%%%%%%%%%%%%%%%

In this section we keep the notations of \cite{Mu}. If $ \Gamma $ is 
a Fuchsian group of the first kind, we denote by $ \mathscr{F}(\Gamma, 
\chi, \kappa,\lambda) $ the space of all non-holomorphic automorphic 
forms of real weight $ \kappa $, multiplier system $ \chi $ and 
Laplace eigenvalue $ \lambda = 1/2-\rho^2 $, $ \Re{\rho}\geq 0 $. 
For our application later we will only need the case that $ \kappa 
= 1/2 $ and $ \rho = 0 $. If $ \infty $ is a cusp of $ \Gamma $, 
the Fourier-Whittaker expansion of the function $ f \in \mathscr{F}
(\Gamma, \chi, \kappa,\lambda) $ at $ \infty $ is given by
\begin{equation*}
\begin{split}
f(W_{\infty}z) &= A_{\infty,0}(y) + \\
& \quad
\sum_{n \neq 0} a_{\infty,n} W_{(\sgn{n})\frac{\kappa}{2},\rho}
(4\pi|n+\mu_{\infty}|y) \exp(2\pi i(n+\mu_{\infty})x).
\end{split}
\end{equation*}
Here $ W_{\infty} $ is the width of the cusp $ \infty $, $ \mu_{\infty} $
is the cusp parameter and $ W_{(\sgn{n})\frac{\kappa}{2}, \rho} $ denotes
the Whittaker function. Furthermore, the constant term of the Fourier
expression has the following form:
\begin{equation}
A_{\infty,0}(y) =
\begin{cases}
0, & \mu_{\infty}\neq 0 \\
a_{\infty,0}y^{1/2+\rho}+b_{\infty,0}y^{1/2+\rho},
& \mu_{\infty}=0,\ \rho\neq 0 \\
a_{\infty,0}y^{1/2}+b_{\infty,0}y^{1/2}\log{y},
& \mu_{\infty}=0,\ \rho= 0.
\end{cases}
\end{equation}
Let $\Gamma$ and $\widehat{\Gamma}$ be two Fuchsian groups of the 
first kind such that $\infty$ is a cusp of both groups. If the 
Fourier-Whittaker expansions of the functions $ f \in \mathscr{F}
(\Gamma, \chi, \kappa, \lambda) $ and $ g \in \mathscr{F}
(\widehat{\Gamma}, \widehat{\chi},\kappa,\lambda)$ at $ \infty $ 
are given by
\begin{equation*} 
\begin{split} 
f(W_{\infty}z) &= A_{\infty,0}(y) + \\ 
& \quad 
\sum_{n \neq 0} a_{\infty,n} W_{(\sgn{n})\frac{\kappa}{2},\rho}(4\pi
|n+\mu_{\infty}|y) \exp(2\pi i(n+\mu_{\infty})x) \\ 
\intertext{and} \\ 
g(\widehat{W}_{\infty}z) &= \widehat{A}_{\infty,0}(y) + \\ 
& \quad 
\sum_{n \neq 0} \widehat{a}_{\infty,n} W_{(\sgn{n})\frac{\kappa}{2},\rho} 
(4\pi|n+\widehat{\mu}_{\infty}|y)\exp(2\pi i(n+\widehat{\mu}_{\infty})x), 
\end{split} 
\end{equation*} 
the associated Dirichlet series
\begin{equation}
L^{+}(f,s) =\sum_{n>0}\frac{a_{\infty,n}}{(n+\mu_{\infty})^s}, \quad
L^{-}(f,s) =\sum_{n<0}\frac{a_{\infty,n}}{|n+\mu_{\infty}|^s}
\end{equation}
are absolutely convergent for $ \Re{s }> \Re{\rho}+1/2 $. 
For a function $ f $ on the upper half-plane $ \HH $ and $ \kappa \in
\R $ the slash operator is defined by 
\begin{equation}\label{eq:slash}
(f\big|_{\kappa}M)(z) =
\left(\frac{cz+d}{|cz+d|} \right)^{-\kappa}f(M z), \quad
M = 
\begin{pmatrix} 
a & b \\ 
c & d  
\end{pmatrix}
\in \textup{SL}_2(\R).
\end{equation}

We prove the following modification of \cite[Theorem~4.1]{Mu}, which is
required for our application.

\begin{thm}\label{thm:muller}
Assume that $ \rho = 0 $. Suppose that there are constants $ C \in \C $
and $ \gamma \in \R^{+} $ such that for all $ z \in \HH $ we have
\begin{equation}\label{eq:trmuller}
\exp(\pi i\kappa/2)(f\big{|}_{\kappa}J)(z) = Cg(\gamma z), \quad
J = \begin{pmatrix} 0 & -1 \\ 1 & 0  \end{pmatrix}.
\end{equation}
Then the Dirichlet series $ L^{\pm}(f,s) $ and $ L^{\pm}(g,s) $ have a
meromorphic continuation to the entire complex plane. Furthermore,
they have only one pole at $ s = 1/2 $. If $ \kappa = -(1+2n) $,
$ n \in \N_0 $, then the pole is simple. In all other cases
$ L^{\pm}(f, s) $ and $  L^{\pm}(g, s) $ have a pole of order $ 2 $
at $ s = 1/2 $.

Moreover, the Dirichlet series $ L^{\pm}(f,s) $ and $ L^{\pm}(g,s) $
satisfy the functional equations
\begin{multline}\label{eq:FELfg}
L^{\pm}(f,s) =
C (2\pi\delta)^{2s} \Gamma^2\left( \frac{1}{2}-s\right) \\
\times \Biggl(\frac{L^{\pm}(g,-s)}{\pi}
\sin{\pi\left(s\pm\frac{\kappa}{2}\right)}
+ \frac{L^{\mp}(g,-s)}{\Gamma^2(\pm \kappa/2+1/2)}\Biggr),
\end{multline}
where
$
\delta = \sqrt{\gamma /(W_{\infty}\widehat{W}_{\infty})}
$.
\end{thm}

\begin{rem}
In the original theorem of M\"{u}ller (see \cite[Theorem~4.1]{Mu}) it
was stated that for $ \rho = 0 $ the functions $ L^{\pm}(f,s) $ have
poles of order one at $ s=1/2 $,
which turns out to be true only for certain values of $ \kappa $.
\end{rem}

\begin{proof}
Equation \eqref{eq:FELfg} is a direct result of \cite[Eq.~44]{Mu} if 
we set $ \rho = 0 $. 
We now prove the  meromorphic continuation of $ L^{+}(f,s) $. Let
\begin{multline}
\Gamma_{\alpha,\rho}(s) :=
2^{\alpha}\frac{\Gamma(s+\rho+1/2)\Gamma(s-\rho+1/2)}{\Gamma(s+1-\alpha)} \\
\times{}_2F_{1}\left(1/2+\rho-\alpha,1/2-\rho-\alpha;s+1-\alpha;1/2 \right).
\end{multline}
First we note that by \cite[Eq.~30,~48]{Mu}
\begin{equation} \label{eq:L+}
L^{+}(f,s) =
\frac{(2\pi)^{s}}{2\Gamma^2(s+1/2)} \left(\Gamma_{1-\kappa/2,0}(s)M(f,s) 
- \Gamma_{-\kappa/2,0}(s)M(E_{\kappa}f,s)\right).
\end{equation}
Here we used the Maa{\ss} lowering operator 
$ 
E_{\kappa}: \mathscr{F}(\Gamma, \chi, \kappa, \lambda) \rightarrow 
\mathscr{F}(\Gamma, \chi, 
$ $ 
\kappa-2, \lambda) 
$ 
which is defined by
\begin{equation}\label{eq:MLO}
E_{\kappa} =
y\left( i\frac{\partial}{\partial x}-\frac{\partial}{\partial y}\right)
+\frac{\kappa}{2}.
\end{equation}
According to \cite[Eq.~52]{Mu} we have
\begin{equation} \label{eq:Mu52}
\begin{split}
M(f,s) &=
C \delta \frac{\widehat{b}_{\infty,0}}{(s-1/2)^2}
+ C \delta\frac{\widehat{a}_{\infty,0}
+ 2 \log{\delta} \, \widehat{b}_{\infty,0}}{s-1/2} \\
& \quad
+ \frac{b_{\infty,0}}{(s+1/2)^2}-\frac{a_{\infty,0}}{s+1/2}+\text{entire function}.
\end{split}
\end{equation}
Furthermore, using \cite[Eq.~53]{Mu} we obtain
\begin{equation*}
\begin{split}
M(E_{\kappa}f,s) &=
\int_{0}^{\infty} \left(E_{\kappa}f(W_{\infty}yi) - E_{\kappa}A_{\infty,0}(y)\right)
y^{s-1}dy \\
&=
-C \delta^{2s} \int_{\delta}^{\infty} \left( E_{\kappa}g(\widehat{W}_{\infty}yi) -
E_{\kappa}\widehat{A}_{\infty,0}(y)\right)y^{-s-1}dy \\
& \quad
+ \int_{\delta}^{\infty} \left(E_{\kappa}f(W_{\infty}yi) 
- E_{\kappa}A_{\infty,0}(y)\right)
y^{s-1}dy \\
& \quad
- C \delta^{2s} \int_{\delta}^{\infty} E_{\kappa} \widehat{A}_{\infty,0}(y) y^{-s-1} dy
- \int_{0}^{\delta}E_{\kappa}A_{\infty,0}(y)y^{s-1}dy.
\end{split}
\end{equation*}
The first two integrals are entire functions.
Applying the Maa{\ss} lowering operator defined by \eqref{eq:MLO}, we get
\begin{equation*}
E_{\kappa}A_{\infty,0}(y) =
a_{\infty,0} \frac{\kappa-1}{2} y^{1/2} + b_{\infty,0}\frac{\kappa-1}{2}y^{1/2}\log{y}
- b_{\infty,0} y^{1/2}.
\end{equation*}
This implies that
\begin{equation} \label{eq:Mekf}
\begin{split}
M(E_{\kappa}f,s) &=
- \frac{C\delta}{(s-1/2)^2}\widehat{b}_{\infty,0}\frac{\kappa-1}{2}
+ \frac{1}{(s+1/2)^2}b_{\infty,0}\frac{\kappa-1}{2} \\
& \quad
- \frac{C\delta}{s-1/2} \left(\widehat{a}_{\infty,0} \frac{\kappa-1}{2}
+ \widehat{b}_{\infty,0}(\kappa-1)\log{\delta} - \widehat{b}_{\infty,0}\right) \\
& \quad
-\frac{1}{s+1/2}\left(a_{\infty,0}\frac{\kappa-1}{2}-b_{\infty,0} \right)
+ \text{entire function}.
\end{split}
\end{equation}
In order to determine the analytic behaviour of $ L^{+}(f, s) $ we
analyze the function
\begin{equation*}
F(s, \kappa) :=
\frac{1}{\Gamma^2(s+\frac{1}{2})} \left(\Gamma_{1-\kappa/2, 0}(s) +
\frac{\kappa - 1}{2} \Gamma_{-\kappa/2, 0}(s)\right).
\end{equation*}
Using \cite[Eq.~22]{Mu} and \cite[Eq.~9.137.11]{GR} this expression can
be simplified to
\begin{equation} \label{F}
\begin{split}
F(s, \kappa) &=
% \frac{2^{1-\kappa/2}}{\Gamma(s+\frac{\kappa}{2})} \hyp\left(\frac{\kappa-1}{2},
% \frac{\kappa-1}{2}; s+\frac{\kappa}{2};\frac{1}{2}\right) \\
% & \quad
% + \frac{\kappa-1}{2} \frac{2^{-\kappa/2}}{\Gamma(s+1+\frac{\kappa}{2})}
% \hyp\left(\frac{\kappa-1}{2}+1, \frac{\kappa-1}{2}+1; s+\frac{\kappa}{2}+1;
% \frac{1}{2}\right) \\
% &=
\frac{2^{1-\kappa/2}}{\Gamma(s+\frac{\kappa}{2})}
\Bigg(\hyp\left(\frac{\kappa-1}{2}, \frac{\kappa-1}{2}; s+\frac{\kappa}{2};
\frac{1}{2}\right) \\
& \quad
+ \frac{\kappa-1}{2} \frac{1}{2(s+\frac{\kappa}{2})} \hyp\left(\frac{\kappa-1}{2}
+1, \frac{\kappa-1}{2}+1; s+\frac{\kappa}{2}+1;\frac{1}{2}\right)\Bigg) \\
&=
\frac{2^{1-\kappa/2}}{\Gamma(s+\frac{\kappa}{2})} \hyp\left(\frac{\kappa-1}{2},
\frac{\kappa+1}{2}; s+\frac{\kappa}{2};\frac{1}{2}\right)
\end{split}
\end{equation}
Thus equations \eqref{eq:L+}, \eqref{eq:Mu52}, \eqref{eq:Mekf} and
\cite[Eq.~22]{Mu} imply
\begin{equation*}
\begin{split}
\lim_{s \rightarrow 1/2} \Big(s-\frac{1}{2}\Big)^2 L^{+}(f, s)
&=
\lim_{s \rightarrow 1/2}
\frac{(2 \pi)^s C \delta \widehat{b}_{\infty,0}}{2} F(s, \kappa) \\
&=
\lim_{s \rightarrow 1/2} \frac{(2 \pi)^s C \delta
\widehat{b}_{\infty,0}}{2^{\kappa/2} \Gamma(s+\frac{\kappa}{2})}
\hyp\left(\frac{\kappa-1}{2}, \frac{\kappa+1}{2}; s+\frac{\kappa}{2};
\frac{1}{2}\right) \\
&=
\frac{(2 \pi)^{1/2} C \delta \widehat{b}_{\infty,0}}{2^{\kappa/2}
\Gamma(\frac{1+\kappa}{2})}
{}_1F_0\left(\frac{\kappa-1}{2}; \frac{1}{2}\right) \\
&=
\frac{(2 \pi)^{1/2} C \delta \widehat{b}_{\infty,0}}{\sqrt{2}
\Gamma(\frac{1+\kappa}{2})}.
\end{split}
\end{equation*}
Therefore, the double pole at $s=1/2$ gets cancelled if and only if
$
1 + \kappa = -2n
$,
$ n \in \N_0 $,
and consequently, the function $L^{+}(f,s)$ has a double pole at $ s= 1/2 $
for $ \kappa \neq -1,-3,-5,-7,\ldots $.
For understanding the behaviour of $ L^{+}(f, s) $ at $ -1/2 $ we look
at the function
\begin{equation*}
G(s, \kappa) := \frac{1}{\Gamma^2(s+\frac{1}{2})}
\left(\Gamma_{1-\kappa/2,0}(s) - \frac{\kappa-1}{2} \Gamma_{-\kappa/2,0}(s)\right).
\end{equation*}
Using \cite[Eq.~9.137.7]{GR} we see that this function is equal to
\begin{equation} \label{G}
\begin{split}
G(s, \kappa) &=
\frac{1}{2^{\kappa/2} \Gamma(s+1+\frac{\kappa}{2})}
\Bigg(2\left(s+\frac{\kappa}{2}\right) \hyp\left(\frac{\kappa-1}{2},
\frac{\kappa-1}{2}; s+\frac{\kappa}{2};\frac{1}{2}\right) \\
&
\qquad
- \frac{\kappa-1}{2} \hyp\left(\frac{\kappa-1}{2}+1,
\frac{\kappa-1}{2}+1; s+\frac{\kappa}{2}+1;\frac{1}{2}\right)\Bigg) \\
&=
\frac{2(s+\frac{1}{2})}{2^{\kappa/2} \Gamma(s+1+\frac{\kappa}{2})}
\hyp\left(\frac{\kappa-1}{2}, \frac{\kappa+1}{2};
s+\frac{\kappa}{2}+1;\frac{1}{2}\right).
\end{split}
\end{equation}
This implies
\begin{equation*}
\begin{split}
\lim_{s \rightarrow -1/2} & \Big(s+\frac{1}{2}\Big) L^{+}(f, s) \\
&=
\lim_{s \rightarrow -1/2} \frac{(2 \pi)^s}{2}
\Bigg(\frac{b_{\infty, 0}}{s+\frac{1}{2}} G(s, \kappa)
-a_{\infty, 0} G(s, \kappa)
- \frac{b_{\infty, 0}}{\Gamma^2(s+\frac{1}{2})} \Gamma_{-\kappa/2,0}(s) \Bigg) \\
% &=
% \lim_{s \rightarrow -1/2} \frac{(2 \pi)^s}{2}
% \Bigg(\frac{2b_{\infty, 0}}{2^{\kappa/2} \Gamma(s+1+\frac{\kappa}{2})}
% \hyp\left(\frac{\kappa-1}{2}, \frac{\kappa+1}{2}; s+\frac{\kappa}{2}+1;
% \frac{1}{2}\right) \\
% & \quad
% - \frac{b_{\infty, 0}}{2^{\kappa/2} \Gamma(s+1+\frac{\kappa}{2})}
% \hyp\left(\frac{\kappa+1}{2}, \frac{\kappa+1}{2}; s+1+\frac{\kappa}{2};
% \frac{1}{2}\right)\Bigg) \\ 
&=
\frac{(2 \pi)^{-1/2}b_{\infty, 0}}{2^{1+\kappa/2} \Gamma(\frac{1+\kappa}{2})}
\Bigg(2 \hyp\left(\frac{\kappa-1}{2}, \frac{\kappa+1}{2}; \frac{1+\kappa}{2};
\frac{1}{2}\right) \\ 
& 
\qquad 
\phantom{\frac{(2 \pi)^{-1/2}b_{\infty, 0}}{2^{1+\kappa/2} 
\Gamma(\frac{1+\kappa}{2})}\Bigg(} 
-
\hyp\left(\frac{\kappa+1}{2}, \frac{\kappa+1}{2}; \frac{1+\kappa}{2};
\frac{1}{2}\right)\Bigg) \\
&=
\frac{(2 \pi)^{-1/2}b_{\infty, 0}}{2^{1+\kappa/2} \Gamma(\frac{1+\kappa}{2})}
\Bigg(2^{\frac{\kappa+1}{2}} - 2^{\frac{\kappa+1}{2}}\Bigg) \\
&= 0
\end{split}
\end{equation*}
and therefore the function $L^{+}(f,s)$ can be analytically continued at
$ s = -1/2 $.

\bigskip

In order to prove the meromorphic continuation of $ L^{-}(f, s) $ we remark
that \cite[Eq.~30,~49]{Mu} gives
\begin{equation*}
\begin{split}
\frac{2\Gamma^2(s+1/2)}{(2\pi)^{s}} & L^{-}(f,s) = \\
&
\left(\frac{\kappa-1}{2}\right)^2\Gamma_{\kappa/2-1,0}(s)M(f,s) +
\Gamma_{\kappa/2,0}(s)M(E_{\kappa}f,s).
\end{split}
\end{equation*}
Using the same approach as before we see that the limit
\begin{equation*}
\begin{split}
\lim_{s \rightarrow 1/2} \frac{\kappa - 1}{\Gamma\left(s+2-\frac{\kappa}{2}\right)}
& \Bigg(\frac{1-\kappa}{2}
\frac{1}{2}
\hyp\left(\frac{3-\kappa}{2},\frac{3-\kappa}{2}; s+2-\frac{\kappa}{2};
\frac{1}{2}\right) \\
&
+ \left(s+1-\frac{\kappa}{2}\right)
\hyp\left(\frac{1-\kappa}{2},\frac{1-\kappa}{2}; s+1-\frac{\kappa}{2};
\frac{1}{2}\right)\Bigg)
\end{split}
\end{equation*}
determines the behaviour of $ L^{-}(s, f) $ at $ s = 1/2 $. By
\cite[Eq.~9.137.11]{GR} this limit is equal to
\begin{equation*}
\begin{split}
(1 - \kappa) \lim_{s \rightarrow 1/2} \frac{s+1-\frac{\kappa}{2}}{\Gamma(s+2-
\frac{\kappa}{2})} \hyp\left(\frac{1-\kappa}{2},\frac{3-\kappa}{2};
s+1-\frac{\kappa}{2};\frac{1}{2}\right)
=
\frac{2^{\frac{k+1}{2}}}{\Gamma\left(\frac{1-\kappa}{2}\right)}.
\end{split}
\end{equation*}
Thus the double pole at $ s = 1/2 $ is cancelled if and only if
\begin{equation*}
1-\kappa = -2n, \quad n \in \N_0,
\end{equation*}
which is true in case of odd positive integral weight, namely for
\begin{equation*}
\kappa = 1, 3, 5, 7, \ldots.
\end{equation*}
Applying \cite[Eq.~9.137.7]{GR} we see infer that
\begin{equation*}
\begin{split}
\lim_{s \rightarrow -1/2} & (s+1/2) L^{-}(f, s) = \\
\lim_{s \rightarrow -1/2} &
\frac{(2\pi)^s 2^{\kappa/2-3} (\kappa-1)}{\Gamma(s+2-\frac{\kappa}{2})}
\Bigg(\frac{\kappa-1}{2} \hyp\left(\frac{3-\kappa}{2},\frac{3-\kappa}{2};
s+2-\frac{\kappa}{2};\frac{1}{2}\right) \\
&
+ 2(s+1-\kappa/2) \hyp\left(\frac{1-\kappa}{2},\frac{1-\kappa}{2};
s+1-\frac{\kappa}{2};1/2\right)\Bigg) \frac{b_{\infty, 0}}{s+1/2} \\
&
+
\frac{2^{\kappa/2-1} b_{\infty, 0}}{(2\pi)^{1/2} \Gamma(\frac{1-\kappa}{2})}
\hyp\left(\frac{1-\kappa}{2},\frac{1-\kappa}{2};
\frac{1-\kappa}{2};\frac{1}{2}\right) \\
&=
\lim_{s \rightarrow -1/2} \frac{(2 \pi)^s 2^{\kappa/2-2} (\kappa-1)
b_{\infty, 0}}{\Gamma(s+2-\frac{\kappa}{2})} \hyp\left(\frac{1-\kappa}{2},
\frac{3-\kappa}{2}; s+2-\frac{\kappa}{2};\frac{1}{2}\right) \\
&
\qquad
+ \frac{2^{\kappa-3/2} b_{\infty, 0}}{(2\pi)^{1/2} \Gamma(\frac{1-\kappa}{2})}.
\end{split}
\end{equation*}
Since
\begin{equation*}
\lim_{s \rightarrow -1/2} \frac{(\kappa-1)}{\Gamma(s+2-\frac{\kappa}{2})}
\hyp\left(\frac{1-\kappa}{2},\frac{3-\kappa}{2}; s+2-\frac{\kappa}{2};1/2\right)
=
- \frac{2^{\frac{\kappa+1}{2}}}{\Gamma\left(\frac{1-\kappa}{2}\right)}
\end{equation*}
this implies that $ L^{-}(f, s) $ can be analytically continued in $ s
= -1/2 $.

\bigskip

According to \cite[Eq.~54]{Mu} we have
\begin{equation}
M(g,s) = M(f,-s)\frac{\delta^{2s}}{C}, \quad
M(E_{\kappa}g,s) = M(E_{\kappa}f,-s)\frac{\delta^{2s}}{C}.
\end{equation}

Then applying \cite[Eq.~30,~48,~49,~54]{Mu} we get
\begin{multline*}
\frac{2\Gamma^2(s+1/2)}{(2\pi)^s}L^{+}(g,s) =
\frac{\delta^{2s}}{C}\biggl( \Gamma_{1-\kappa/2,0}(s)M(f,-s) \\
+\Gamma_{-\kappa/2,0}(s)M(E_{\kappa}f,-s)\biggr)
\end{multline*}
and
\begin{multline*}
\frac{2\Gamma^2(s+1/2)}{(2\pi)^s}L^{-}(g,s) =
\frac{\delta^{2s}}{C} \biggl(\frac{(\kappa-1)^2}{4}
\Gamma_{\kappa/2-1,0}(s)M(f,-s) \\
-\Gamma_{\kappa/2,0}(s)M(E_\kappa f,-s)\biggr).
\end{multline*}
Similarly to the case of $ L^{\pm}(f,s) $ it follows that the functions 
$ L^{\pm}(g, s) $ have a double pole at $ s=1/2 $ unless $ \kappa $ is 
an odd negative integer.
\end{proof}

%%%%%%%%%%%%%%%%%%%%%%%%%%%%%%%%%%%%%%%%%%%%%%%%%%%%%%%%%%%%%%%%%%%%%%
\subsection{Some transformation properties}\label{transf}
%%%%%%%%%%%%%%%%%%%%%%%%%%%%%%%%%%%%%%%%%%%%%%%%%%%%%%%%%%%%%%%%%%%%%%

The main references for this section are \cite{K,Pit}.
Consider the function  $f \in \mathscr{F}(\Gamma_0(4), \chi, 1/2, 
\lambda) $ with the Fourier-Whittaker expansion
\begin{equation}\label{eq:FEfz}
f(z)=A_0(y)+\sum_{n \neq 0}a_nW_{\sgn(n)/4,\rho}(4\pi |n|y)\exp(2\pi inx),
\end{equation}
where $z=x+iy \in \HH.$
Assume that
\begin{equation}\label{eq:plusspace}
a_n=0\text{ for } n \equiv 2,3 \pmod{4}.
\end{equation}
Furthermore, we define the following operators
\begin{equation}
(f\big|U)(z) = 
\frac{1}{\sqrt{2}} \left(f\left(\frac{z}{4}\right) 
+ f\left(\frac{z+2}{4}\right)\right),
\end{equation}
\begin{equation} \label{eq:trandfW}
(f\big|W)(z) = 
\left(\frac{-iz}{|z|}\right)^{-1/2} f\left(\frac{-1}{4z}\right).
\end{equation}
These operators leave the space $ \mathscr{F}(\Gamma_0(4), \chi, 1/2, 
\lambda)$ stable.
If $ f $ has the expansion \eqref{eq:FEfz}, then the function 
\begin{equation}\label{eq:defg}
g(z):=\frac{1}{\sqrt{2}}(f\big|U)(z)
\end{equation}
satisfies
\begin{equation}
g(z) = 
A_0(y/4)+\sum_{n \neq 0}a_{4n}W_{\sgn(n)/4,\rho}(4\pi |n|y)\exp(2\pi inx).
\end{equation}
By \cite[Prop. ~4.1]{Pit} we have
\begin{equation}\label{eq:tranfuw}
(f\big|U\big|W)(z)=f(z).
\end{equation}

\begin{lem}\label{lem:transf} 
Let $ f $ and $ g $ be as above. For $ z = x+iy \in \HH$ we have
\begin{equation}
\exp{(\pi i/4)} \left(f\big|_{1/2}J\right)(z) = \sqrt{2}g(z/4), 
\quad J = 
\begin{pmatrix} 0 & -1 \\ 1 & 0  
\end{pmatrix}.
\end{equation}
\end{lem}
\begin{proof}
Applying equations \eqref{eq:trandfW}, \eqref{eq:tranfuw} and 
\eqref{eq:slash}, we find
\begin{multline*}
\sqrt{2}g(z/4) = (f\big|U)(z/4) 
= i^{1/2}\left(\frac{z}{|z|} \right)^{-1/2}(f\big|U\big|W)(-1/z) =\\
= i^{1/2}\left(\frac{z}{|z|} \right)^{-1/2}f(-1/z) 
= \exp{(\pi i/4)}\left(f\big|_{1/2}J\right)(z).
\end{multline*}
\end{proof}

%%%%%%%%%%%%%%%%%%%%%%%%%%%%%%%%%%%%%%%%%%%%%%%%%%%%%%%%%%%%%%%%%%%%%%
\subsection{Eisenstein series of weight $1/2$ and level $4$}
%%%%%%%%%%%%%%%%%%%%%%%%%%%%%%%%%%%%%%%%%%%%%%%%%%%%%%%%%%%%%%%%%%%%%%

The connection between quadratic Dirichlet $L$-functions and Fourier
coefficients of half-integral weight forms was first discovered by
Maa{\ss} in 1937 \cite{M}.  Similar results were obtained by Shimura
\cite{S1}, Shintani \cite{Sh}, Cohen \cite{C}, Goldfeld-Hoffstein
\cite{GH}, and others.

We study the more general Dirichlet series (see \cite{B} for details)
\begin{equation}
\mathscr{L}_n(s) = \frac{\zeta(2s)}{\zeta(s)} \sum_{q=1}^{\infty}
\frac{1}{q^s} \left(\sum_{1\leq t \leq 2q;t^2 \equiv n \Mod{4q}}1\right).
\end{equation}
Only if $ n \equiv 0,1 \Mod{4} $ the function $ \mathscr{L}_n(s) $ 
considered as a function of $ s $ does not vanish.
We can think of $\mathscr{L}_n(s)$ as a certain generalization of the
Riemann zeta function and quadratic Dirichlet $L$-functions. Indeed,
\begin{equation*}
\mathscr{L}_0(s)=\zeta(2s-1).
\end{equation*}
If $ n $ is a fundamental discriminant, then
\begin{equation*}
\mathscr{L}_n(s) = \sum_{q=1}^{\infty}\frac{\chi_n(q)}{q^s},
\end{equation*}
where $ \chi_n $ is a primitive quadratic character mod $ |n| $.
For any $ \epsilon>0 $
\begin{equation}\label{eq:subcL}
\mathscr{L}_{n}(1/2)\ll n^{\theta+\epsilon},
\end{equation}
where $ \theta $ is a subconvexity exponent for Dirichlet $ L $-functions.
The best known result $ \theta = 1/6 $ is due to Conrey and Iwaniec 
\cite{CI}. The Lindel\"{o}f hypothesis asserts that $\theta=0$.
The completed $L$-function
\begin{equation*}
\mathscr{L}_{n}^{*}(s) = 
(\pi/|n|)^{-s/2} \Gamma(s/2+1/4-\sgn{n}/4) \mathscr{L}_n(s)
\end{equation*}
satisfies the functional equation
\begin{equation}
\mathscr{L}_{n}^{*}(s)=\mathscr{L}_{n}^{*}(1-s).
\end{equation}
The function $\mathscr{L}_{n}^{*}(s)$ appears in the Fourier-Whittaker
expansion of the combination of the Maa{\ss}-Eisenstein series of weight
$ 1/2 $ and level $ 4 $ at the cusps $ \infty $ and $ 0 $, namely
\begin{equation}\label{lk}
E^{(1/2)}_{(0; \infty)}(z;s) = \zeta(4s-1) 
\left(E_{\infty}(z;s;1/2) + \frac{1+i}{4^s} E_{0}(z;s;1/2)\right),
\end{equation}
where for a cusp $ \alpha $ the series $E_{\alpha}(z;s;k)$ is defined 
in \cite[Section~3]{PRR}.

\begin{lem}\label{fourexp}
We have
\begin{equation*} 
\begin{split} 
& E^{(1/2)}_{(0; \infty)}(z;s) = 
\zeta(4s-1)y^s + \frac{\sqrt{\pi}\Gamma(2s-1) \zeta(4s-2)}{4^{2s-1}
\Gamma(2s-1/2)}y^{1-s} + \frac{\pi^{2s-3/4}}{4^s\Gamma(2s-1/2)} \\ 
& \quad \phantom{E^{(1/2)}_{(0; \infty)}(z;s) =} 
\times \sum_{n\neq 0}\frac{\mathscr{L}_{n}^{*}(2s-1/2)}{|n|^{3/4}}
W_{\sgn{n}/4,s-1/2}\left(4\pi|n|y\right)\exp{(2\pi i nx)}.
\end{split} 
\end{equation*} 
\end{lem}

\begin{proof}
The Fourier-Whittaker expansion for $ E_{\infty}(z;s;1/2) $ is given in
Section $3.1$ of \cite{PRR}. Computations for $E_0(z;s;1/2)$ are similar.
\end{proof}

\noindent 
Computing the limit as $s\rightarrow 1/2$, we find
\begin{multline}
E^{(1/2)}_{(0; \infty)}(z;1/2) = 
\frac{1}{2}y^{1/2}\log{y} + (\gamma-\log{4\pi})y^{1/2}
+\frac{1}{2\sqrt{\pi}} \times \\
\times \sum_{n\neq 0} \frac{\mathscr{L}_{n}(1/2)}{|n|^{1/2}}
\Gamma\left(\frac{1}{2}-\frac{\sgn{n}}{4}\right) 
W_{\sgn{n}/4,0}\left(4\pi|n|y\right) \exp{(2\pi i n x)}.
\end{multline}

Now we can apply the results of Sections \ref{muller} and \ref{transf}.
Note that condition \eqref{eq:plusspace} is satisfied since
$\mathscr{L}_n(1/2)$ vanishes for $n \equiv 2,3 \Mod{4}.$
For $$f(z):=E^{(1/2)}_{(0; \infty)}(z;1/2)$$ define $g(z)$ by equation
\eqref{eq:defg}. Then $ g $ has the expansion 
\begin{multline}
g(z) = \frac{1}{4}y^{1/2}\log{y} + \frac{1}{2}(\gamma-\log{8\pi})y^{1/2}
+ \frac{1}{4\sqrt{\pi}} \times \\
\times \sum_{n\neq 0} \frac{\mathscr{L}_{4n}(1/2)}{|n|^{1/2}}
\Gamma\left(\frac{1}{2}-\sgn{n}\right)
W_{\sgn{n}/4,0}\left(4\pi|n|y\right)\exp{(2\pi i nx)}.
\end{multline}

Now we consider the associated Dirichlet series
\begin{equation}
L^{+}_{f}(s) =
\frac{\Gamma(1/4)}{2\sqrt{\pi}}\sum_{n>0}\frac{\mathscr{L}_{n}(1/2)}{n^{s+1/2}}, 
\quad
L^{-}_{f}(s) =
\frac{\Gamma(3/4)}{2\sqrt{\pi}}\sum_{n<0}\frac{\mathscr{L}_{n}(1/2)}{|n|^{s+1/2}},
\end{equation}
\begin{equation}\label{eq:lg}
L^{+}_{g}(s)=\frac{-1}{2}\sum_{n>0}\frac{\mathscr{L}_{4n}(1/2)}{n^{s+1/2}},
\quad
L^{-}_{g}(s)=\frac{1}{8}\sum_{n<0}\frac{\mathscr{L}_{4n}(1/2)}{|n|^{s+1/2}}.
\end{equation}

Note that by Lemma \ref{lem:transf} equation \eqref{eq:trmuller} is satisfied
for $ \kappa = 1/2 $, $ C = \sqrt{2}$, $ \gamma = 1/4$, and therefore, we can 
apply Theorem \ref{thm:muller}.

\begin{thm}\label{thm:lfuncteq}
The functions $L^{\pm}_{f}(s)$ and $L^{\pm}_{g}(s)$ have a meromorphic
continuation to the whole complex plane and satisfy the functional
equations
\begin{multline}
L^{+}_{g}(s) =
\frac{-\pi^{2s+2}}{\sqrt{2}\Gamma^2(1/2+s)\sin^{2}{\pi s}} \times \\
\left(\frac{\sin \pi(-s-1/4)}{\pi}L^{+}_{f}(-s)
- \frac{L^{-}_{f}(-s)}{\Gamma^2(3/4)} \right),
\end{multline}
\begin{multline}\label{eq:lgs}
L^{-}_{g}(s) = \frac{\pi^{2s+2}}{\sqrt{2}\Gamma^2(1/2+s)\sin^{2}{\pi s}}
\times \\
\left(-\frac{\sin \pi(-s+1/4)}{\pi} L^{-}_{f}(-s)
+ \frac{L^{+}_{f}(-s)}{\Gamma^2(1/4)}\right). 
\end{multline}
Furthermore, $L^{\pm}_{f}(s)$ and $L^{\pm}_{g}(s)$ are holomorphic in $\C$
except for a double pole at $s=1/2$.
\end{thm}

%%%%%%%%%%%%%%%%%%%%%%%%%%%%%%%%%%%%%%%%%%%%%%%%%%%%%%%%%%%%%%%%%%%%%%
\section{Preliminary evaluation}
%%%%%%%%%%%%%%%%%%%%%%%%%%%%%%%%%%%%%%%%%%%%%%%%%%%%%%%%%%%%%%%%%%%%%%

Let $ \HH $ be the Poincare upper half-plane and denote by $H_{2k}$ 
the normalized Hecke basis for the space of holomorphic cusp forms 
of even weight $2k \geq 2$ with respect to the full modular group. 
If the function $ f \in H_{2k} $ has the Fourier expansion 
\begin{equation}
f(z)=\sum_{n\geq 1}\lambda_f(n)n^{k-1/2}\exp(2\pi inz), 
\quad \lambda_f(1)=1, 
\end{equation}
the associated Hecke $ L $-function is defined by
\begin{equation}
L(f,s)=\sum_{n \geq 1}\frac{\lambda_f(n)}{n^s}, \quad \Re{s}>1.
\end{equation} 
Let $\Gamma(s)$ be the Gamma function. Then the completed 
$ L $-function 
\begin{equation}
\Lambda(f, s) 
= \left(\frac{1}{2 \pi}\right)^s \Gamma\left(s+\frac{2k-1}{2}\right) 
L(f, s)
\end{equation} 
satisfies the functional equation
\begin{equation} \label{eq: functionalE}
\Lambda(f, s)=\epsilon_f\Lambda(f, 1-s), \quad \epsilon_f=i^{2k}, 
\end{equation}
and can be analytically continued to the entire complex plane. Note 
that by equation \eqref{eq: functionalE} we have $ L_f(1/2) = 0 $ 
for odd $ k $.

For $\Re{s}>1$ the symmetric square $ L $-function is defined by
\begin{equation}
L(\sym^2f,s)=\zeta(2s)\sum_{n=1}^{\infty}\frac{\lambda_f(n^2)}{n^s}.
\end{equation}
Let 
\begin{equation}
L_{\infty}(s) := 
\pi^{-3s/2} \Gamma\left(\frac{s+1}{2}\right) 
\Gamma\left(\frac{s-1}{2}+k\right) \Gamma\left(\frac{s}{2}+k\right).
\end{equation}
The completed $L$-function
\begin{equation*}
\Lambda(\sym^2f,s):=L_{\infty}(s)L(\sym^2f,s)
\end{equation*}
is entire and satisfies the functional equation
\begin{equation}
\Lambda(\sym^2f,s)=\Lambda(\sym^2f,1-s).
\end{equation}

In this section we combine the techniques of analytic continuation 
and the approximate functional equation in order to express the 
average
\begin{equation}
\sum_{f\in H_{4k}}\omega(f)L(f,1/2)L(\sym^2f, 1/2)
\end{equation}
as a sum of three parts.

First, we use the exact formula for the twisted moment of symmetric 
square $L$-functions.
\begin{lem}\label{lem:EF}
We have
\begin{equation*}
\sum_{f \in H_{4k}} \omega(f) \lambda_f(l) L(\sym^2f, 1/2+u) 
= 
M^{D}(u,l) \delta_{l=\Box} + M^{ND}(u,l) + ET(u,l),
\end{equation*}
where
\begin{equation}
\delta_{l=\Box} = 
\begin{cases} 
1 & \text{if } l \text{ is a full square,} \\ 
0 & \text{otherwise}, 
\end{cases}
\end{equation}
\begin{multline} \label{eq:MT}
M^D(u,l^2) 
= 
\frac{\zeta(1+2u)}{l^{1/2+u}} + \sqrt{2}(2\pi)^{3u}\cos{\pi(1/4+u/2)} 
\times \\ 
\frac{\zeta(1-2u)}{l^{1/2-u}} \frac{\Gamma(2k-1/4-u/2) \Gamma(2k+1/4-u/2) 
\Gamma(1-2u)}{\Gamma(2k+1/4+u/2) \Gamma(2k-1/4+u/2)\Gamma(1-u)}, 
\end{multline}
\begin{equation}
M^{ND}(0,l) 
= 
\frac{\sqrt{2\pi}}{2l^{1/4}} \frac{\Gamma(2k-1/4)}{\Gamma(2k+1/4)} 
\mathscr{L}_{-4l}(1/2),
\end{equation}
\begin{multline}\label{eq:ET}
ET(0,l) =
\frac{1}{l^{1/4}} \sum_{1\leq n<2\sqrt{l}} \mathscr{L}_{n^2-4l}(1/2) 
\Phi_{2k}\left(\frac{n^2}{4l}\right) + \\
\frac{1}{\sqrt{2}l^{1/2}} \sum_{n>2\sqrt{l}} \mathscr{L}_{n^2-4l}(1/2) 
\sqrt{n} \Psi_{2k}\left(\frac{4l}{n^2} \right).
\end{multline}
\end{lem}
\begin{proof}
See \cite[Eq.~2.9,~5.6]{BF1}.
\end{proof}
\begin{rem}
The role of the shift $u$ is to simplify the evaluation of the diagonal 
main term in Section \ref{sec:MT}.
\end{rem}
Second, we obtain an approximate functional equation for the Hecke 
$ L $-function at the central point.
Let 
\begin{equation*}
g(s,u) := 
\frac{\left(s^2-(-1/4-u/2)^2\right) 
\left( s^2-(-1/4+u/2)^2\right)}{(1/4+u/2)^2(1/4-u/2)^2}.
\end{equation*}

\begin{lem}\label{lem:AFE} We have
\begin{equation}
L(f,1/2)=2\sum_{l=1}^{\infty}\frac{\lambda_f(l)}{\sqrt{l}}V_k(l,u) 
\end{equation}
where 
\begin{equation}
V_k(l,u) = 
\frac{1}{2\pi i}\int_{(\sigma)}g(s,u)\frac{\Gamma(k+s)}{\Gamma(k)} 
\frac{ds}{(2\pi l)^s s}, \ \sigma > 0. 
\end{equation} 
\end{lem}

\begin{proof}
Consider
\begin{equation*}
I(s)=\int_{(\sigma)}g(s,u)\Lambda(f,1/2+s)\frac{ds}{s}, \quad \sigma>0.
\end{equation*}
Moving the contour of integration from $\sigma$ to $-\sigma$ we pick 
up a simple pole at $s=0$, obtaining
\begin{equation*}
2I(s)=\Lambda(f,1/2).
\end{equation*}
The assertion follows.
\end{proof}
As a consequence of Lemma~\ref{lem:EF} and Lemma~\ref{lem:AFE} we 
get the following decomposition: 
\begin{multline} \label{eq:product} 
\sum_{f \in H_{4k}} \omega(f) L(f,1/2) L(\sym^2f, 1/2+u) = 
2 \sum_{l=1}^{\infty} \frac{V_k(l,u)}{\sqrt{l}} M^{D}(u,l) \delta_{l=\Box} \\ 
+ 2 \sum_{l=1}^{\infty} \frac{V_k(l,u)}{\sqrt{l}} M^{ND}(u,l) 
+ 2 \sum_{l=1}^{\infty} \frac{V_k(l,u)}{\sqrt{l}}ET(u,l).
\end{multline}

%%%%%%%%%%%%%%%%%%%%%%%%%%%%%%%%%%%%%%%%%%%%%%%%%%%%%%%%%%%%%%%%%%%%%%%%%%%%%%%%%%%%%%%%%%%%%%%%%%%%%%%%%%%%%%%%%%

\section{Diagonal main term}\label{sec:MT}

In this section we evaluate asymptotically the diagonal term in 
\eqref{eq:product}, namely
\begin{equation*}
M^{D}:=2\sum_{l=1}^{\infty}\frac{V_k(l,0)}{\sqrt{l}}M^{D}(0,l)\delta_{l=\Box}.
\end{equation*}

\begin{lem}
For any $\epsilon>0$ and any real number $0<a<k$ we have
\begin{multline}
2\sum_{l=1}^{\infty}\frac{V_k(l,u)}{\sqrt{l}}M^{D}(u,l)\delta_{l=\Box} = 
2\Biggl( \zeta(1+2u)\zeta(3/2+u)+\\
\sqrt{2}(2\pi)^{3u}\cos{\pi(1/4+u/2)}\zeta(1-2u)\zeta(3/2-u) \times \\ 
\frac{\Gamma(2k-1/4-u/2)}{\Gamma(2k+1/4+u/2)}
\frac{\Gamma(2k+1/4-u/2)\Gamma(1-2u)}{\Gamma(2k-1/4+u/2)\Gamma(1-u)}\Biggr) 
+ O(k^{-a+\epsilon}).
\end{multline}
\end{lem}
\begin{proof}
Consider
\begin{equation*}
2\sum_{l=1}^{\infty}\frac{V_k(l,u)}{\sqrt{l}}M^D(u,l)\delta_{l=\Box} = 
\frac{2}{2\pi i}\int_{(\sigma)}\frac{\Gamma(k+s)}{\Gamma(k)} 
\frac{g(s,u)}{(2\pi)^s}
\sum_{l=1}^{\infty}\frac{M^{D}(u,l^2)}{l^{2s+1}}\frac{ds}{s}.
\end{equation*}
According to \eqref{eq:MT} this is equal to
\begin{multline*} 
\frac{2}{2\pi i}\int_{(\sigma)}\frac{\Gamma(k+s)}{\Gamma(k)}g(s,u)
\Biggl[\frac{\zeta(1+2u)\zeta(3/2+u+2s)}{2} + \\  
\sqrt{2}(2\pi)^{3u}\cos{\pi(1/4+u/2)}\zeta(1-2u)\zeta(3/2-u+2s)\times \\ 
\frac{\Gamma(2k-1/4-u/2)\Gamma(2k+1/4-u/2)\Gamma(1-2u)}{\Gamma(2k+1/4+u/2) 
\Gamma(2k-1/4+u/2)\Gamma(1-u)}\Biggr]\frac{ds}{(2\pi)^ss}.
\end{multline*}
The integrand has poles at
\begin{equation*}
s=0, \quad s=-k-j, \quad j=0,1,\ldots
\end{equation*}
Note that the poles at $s=-1/4\pm u/2$ are compensated by the zeros of 
$g(s,u)$. Crossing the pole at $s=0$, we can move the line of integration 
to any real number $-a$ such that $0<a<k$. 
Consequently, the resulting integral is bounded by $O(k^{-a+\epsilon})$. 
The assertion follows by calculating the residue at $s=0$.
\end{proof}

Computing the limit as $u \rightarrow 0$, we evaluate the main term at 
the central point.
\begin{cor}\label{cor:MT1}
For any $\epsilon>0$ and any real number $0<a<k$ we have
\begin{equation} 
\begin{split} 
M^{D} = 
\zeta(3/2) & \Biggl(-3 \log{2\pi} + \frac{\pi}{2} + 3\gamma + 
2 \frac{\zeta'(3/2)}{\zeta(3/2)} \\ 
& \quad 
+ \psi(2k-1/4)+\psi(2k+1/4)\Biggr)
+ O(k^{-a+\epsilon}), 
\end{split} 
\end{equation}
where $\psi(x)$ is the logarithmic derivative of the Gamma function. 
\end{cor}

Finally, the diagonal main term can be averaged over the weight $k$ with 
a suitable test function.
\begin{cor}\label{cor:diagonal} 
For any fixed $A>0$ we have
\begin{multline}
\sum_{k}h\left( \frac{4k}{K}\right)M^{D}=\frac{HK}{4}\zeta(3/2)\times \\
\Biggl(2\log{K}-3\log{2\pi}-2\log{2}+\frac{\pi}{2}+3\gamma 
+ 2\frac{\zeta'(3/2)}{\zeta(3/2)}+2\frac{H_1}{H} \Biggr) \\ 
+P_h(1/K)+O(K^{-A}),
\end{multline}
where
\begin{equation}
H=\int_{0}^{\infty}h(y)dy, \quad H_1=\int_{0}^{\infty}h(y)\log{y}dy
\end{equation}
and $P_h(x)$ is a polynomial in $x$ of degree $A-1$ and with coefficients 
depending on the function $h$.
\end{cor}

\begin{proof}
This is derived from Corollary \ref{cor:MT1} by using the asymptotic formula
\begin{equation*}
\psi(2k-1/4) + \psi(2k+1/4) = 
2\log{k} + 2\log{2} + P(1/k) + O\left(\frac{1}{k^{A+1}}\right),
\end{equation*}
where $P(x)$ is a polynomial of degree $A$ such that $P(0)=0$,
and results of \cite[Section~7]{BF}, namely
\begin{equation*}
\sum_{k} h\left(\frac{4k}{K}\right) = 
\frac{HK}{4}+O\left( \frac{1}{K^{b}}\right) 
\end{equation*}
and 
\begin{equation*}
\sum_{k} h\left(\frac{4k}{K}\right)\log{k} = 
\frac{HK}{4}(\log{K}-\log{4})+\frac{H_1K}{4}+O\left( \frac{1}{K^{b}}\right)
\end{equation*}
for any $b>0$.
\end{proof}

%%%%%%%%%%%%%%%%%%%%%%%%%%%%%%%%%%%%%%%%%%%%%%%%%%%%%%%%%%%%%%%%%%%%%%%%%%%%%%%%%%%%%%%%%%%%%%%%%%%%%%%%%%%%%%%%%%

\section{Non-diagonal main term}\label{sec:MT2}

Consider the non-diagonal term
\begin{equation*} 
\begin{split} 
M^{ND} &:= 2 \sum_{l=1}^{\infty} \frac{V_{k}(l,0)}{\sqrt{l}}M^{ND}(0,l) \\ 
&= 
\sqrt{2\pi} \frac{\Gamma(2k-1/4)}{\Gamma(2k+1/4)} 
\sum_{l=1}^{\infty}\frac{V_{k}(l,0)}{l^{3/4}}\mathscr{L}_{-4l}(1/2).
\end{split} 
\end{equation*}

\begin{lem}\label{nondiag}
For any $A>0$ we have
\begin{equation}
M^{ND}= 8\sqrt{2\pi} L^{-}_{g}(1/4) \frac{\Gamma(2k-1/4)}{\Gamma(2k+1/4)}+O(k^{-A}).
\end{equation}
\end{lem}
\begin{proof}
Consider
\begin{equation*}
\sum_{l=1}^{\infty}\frac{V_k(l,0)}{l^{3/4}}\mathscr{L}_{-4l}(1/2) 
= \frac{1}{2\pi i}\int_{(\sigma_1)}\frac{\Gamma(k+s)}{\Gamma(k)} 
\sum_{l=1}^{\infty}\frac{\mathscr{L}_{-4l}(1/2)}{l^{s+3/4}} 
\frac{g(s,0)ds}{s(2\pi)^s}.
\end{equation*}
We assume that $\sigma_1>1$ to justify the change of order of summation 
and integration.
Equation \eqref{eq:lg} implies that
\begin{equation*}
\frac{1}{8}\sum_{l=1}^{\infty}\frac{\mathscr{L}_{-4l}(1/2)}{l^{s+3/4}}
=L^{-}_{g}(s+1/4).
\end{equation*}
According to Theorem \ref{thm:lfuncteq} the function $L^{-}_{g}(s+1/4)$
is holomorphic in $\C$ except for a double pole at $s=1/4$, which is 
compensated by the zeros of $g(s,0)$. Therefore, moving the contour
of integration from $\sigma_1$ to $\sigma_2=-A$ for any $A>0$, we cross
only a simple pole at $s=0$.  Consequently,
\begin{equation*}
\sum_{l=1}^{\infty}\frac{V_k(l,0)}{l^{3/4}}\mathscr{L}_{-4l}(1/2) 
= \res_{s=0}F(s)+\frac{1}{2\pi i}\int_{-A}F(s)ds,
\end{equation*}
where
\begin{equation*}
F(s):=8g(s,0)\frac{\Gamma(k+s)}{\Gamma(k)}\frac{L^{-}_{g}(s+1/4)}{(2\pi)^ss}.
\end{equation*}
The functional equation (see \eqref{eq:lgs})
\begin{equation*}
\begin{split} 
L^{-}_{g}(s+1/4) &= 
\frac{\pi^{2s+5/2}}{\sqrt{2} \Gamma^2(s+3/4) \sin^2\pi(s+1/4)} \times \\ 
& \quad 
\left(\frac{1}{\Gamma^2(1/4)} L^{+}_{f}(-s-1/4) 
+ \frac{\sin\pi s}{\pi}L^{-}_{f}(-s-1/4)\right)
\end{split} 
\end{equation*}
and the estimate
\begin{equation*}
\frac{\Gamma(k+s)}{\Gamma(k)}\ll \frac{\exp(-\pi|s|/2)}{k^A}
\end{equation*}
imply that
\begin{equation*}
\int_{-A} \frac{\Gamma(k+s)}{\Gamma(k)} L^{-}_{g}(s+1/4) \frac{g(s,0)ds}{(2\pi)^ss} 
\ll k^{-A}.
\end{equation*}
 The residue at the origin is equal to
\begin{equation*}
\res_{s=0}F(s)=8L^{-}_{g}(1/4).
\end{equation*}
Therefore, for any $A>0$ we have
\begin{equation*}
\sum_{l=1}^{\infty}\frac{V_k(l,0)}{l^{3/4}}\mathscr{L}_{-4l}(1/2) 
= 8L^{-}_{g}(1/4)+O(k^{-A}).
\end{equation*}
The assertion follows.

\end{proof}
\begin{cor}\label{cor:nondiagonal} For any fixed $A>0$ we have
\begin{equation}
\sum_{k} h\left(\frac{4k}{K}\right) M^{ND} 
= L^{-}_{g}(1/4) \sqrt{K}Q_{h}(1/K)+O(K^{-A}),
\end{equation}
where $Q_h(x)$ is a polynomial in $x$ of degree $A$ with coefficients depending 
on the function $h$.
\end{cor}
\begin{proof}
The statement follows from Lemma \ref{nondiag} and \cite[Eq~5.11.13]{HMF}.
\end{proof}

%%%%%%%%%%%%%%%%%%%%%%%%%%%%%%%%%%%%%%%%%%%%%%%%%%%%%%%%%%%%%%%%%%%%%%%%%%%%%%%%%%%%%%%%%%%%%%%%%%%%%%%%%%%%%%%%%%

\section{Error terms}

Finally, we estimate the last term appearing in (\ref{eq:product}) 
which we split into two terms 
\begin{equation*}
ET:=2\sum_{l=1}^{\infty}\frac{V_k(l,0)}{\sqrt{l}}ET(0,l)=ET_1+ET_2. 
\end{equation*}
By equation \eqref{eq:ET} we have
\begin{equation}\label{eq:ET3}
ET_1 = 
\sqrt{2} \sum_{l=1}^{\infty} \frac{V_k(l,0)}{l} \sum_{n>2\sqrt{l}} 
\mathscr{L}_{n^2-4l}(1/2)\sqrt{n} \Psi_{2k}\left(\frac{4l}{n^2}\right),
\end{equation}
\begin{equation}\label{eq:ET2}
ET_2 = 
2 \sum_{l=1}^{\infty} \frac{V_k(l,0)}{l^{3/4}} \sum_{n<2\sqrt{l}} 
\mathscr{L}_{n^2-4l}(1/2) \Phi_{2k}\left(\frac{n^2}{4l}\right).
\end{equation}

\noindent 
Note that if $l>k^{1+\epsilon}$, then
\begin{equation}\label{eq:Vk}
V_k(l,0)\ll \frac{1}{k^{A}} \text{ for any } A>0.
\end{equation}

Using inequality \eqref{eq:Vk}  we can assume that  $l<k^{1+\epsilon}$ 
in \eqref{eq:ET3} and \eqref{eq:ET2} at the cost of a negligible error 
term.

%%%%%%%%%%%%%%%%%%%%%%%%%%%%%%%%%%%%%%%%%%%%%%%%%%%%%%%%%%%%%%%%%%%%%%%%%%%%%%%%%%%%%%%%%%%%%%%%%%%%%%%%%%%%%%%%%%

\subsection{Error term of the first type}

\begin{lem}\label{lem:firsttype} 
For any $ A > 0 $ we have 
\begin{equation}
\sum_{k}h\left( \frac{4k}{K}\right)ET_1\ll K^{-A}.
\end{equation}
\end{lem}

\begin{proof}
It follows from \cite[Lemma~7.3]{BF1} that
\begin{equation*}
\sum_{k} h\left(\frac{4k}{K}\right) ET_1 
\ll \sum_{k} h\left(\frac{4k}{K}\right) \sum_{l<k^{1+\epsilon}} \frac{1}{\sqrt{l}} 
\frac{l^{-1/24}}{\sqrt{k}} \exp\left(-c\frac{k}{l^{1/4}}\right) \ll K^{-A}.
\end{equation*}
\end{proof}

%%%%%%%%%%%%%%%%%%%%%%%%%%%%%%%%%%%%%%%%%%%%%%%%%%%%%%%%%%%%%%%%%%%%%%%%%%%%%%%%%%%%%%%%%%%%%%%%%%%%%%%%%%%%%%%%%%

\subsection{Error term of the second type}

First, we prove a simple case corresponding to $N=1$ in the Liouville-Green 
approximation and then we explain how to modify the arguments in order to 
sharpen the estimate to $K^{-A}$.
\begin{lem} \label{lem:secondtype} 
We have 
\begin{equation}
\sum_{k} h\left(\frac{4k}{K}\right) ET_2 \ll K^{1/4+\theta}.
\end{equation}
\end{lem}

\begin{proof}
We decompose the sum over $l$ in \eqref{eq:ET2} into two parts:
\begin{equation*}
\sum_{l=1}^{\infty}=\sum_{l\neq \Box}+\sum_{l=\Box}.
\end{equation*}

Suppose that $l$ is not a full square. To approximate the function $\Phi_{2k}$ 
we apply \cite[Theorem~6.10]{BF1} with
\begin{equation}\label{eq:xi}
\cos^2{\sqrt{\xi}}:=n^2/(4l),
\end{equation} 
namely
\begin{equation*}
\Phi_{2k}(\cos^2\sqrt{\xi}) = 
\frac{-\pi}{\xi^{1/4}(\sin{\sqrt{\xi}})^{1/2}} \left[(1+C_J)Z_J(\xi) 
+ C_YZ_Y(\xi)\right].
\end{equation*}
Using \cite[Theorem~6.5]{BF1} to evaluate  $Z_J(\xi)$, $Z_Y(\xi)$ and 
\cite[Corollary~6.9]{BF1} to approximate the constants $C_J$, $C_Y$,
we have
\begin{multline}\label{eq:Phi2k}
\Phi_{2k}(\cos^2\sqrt{\xi}) = 
\frac{-\pi}{\xi^{1/4}(\sin{\sqrt{\xi}})^{1/2}} 
\Biggl[\sqrt{\xi}Y_0((4k-1)\sqrt{\xi}) + \\ 
\sqrt{\xi}J_0((4k-1)\sqrt{\xi}) 
+ O\left(\frac{1}{k}\left|\sqrt{\xi}Y_0((4k-1)\sqrt{\xi})\right|\right)\Biggr].
\end{multline}
In order to apply \cite[Eq.~10.7.8]{HMF} to approximate the Bessel 
functions above we show that the argument is large, namely
\begin{equation*}
((4k-1)\sqrt{\xi})\gg k^{1/2-\epsilon}.
\end{equation*}
Indeed, since $l$ is not a full square there exist $u$ and $m$ such 
that $4l=m^2+u$, $1\leq u\leq 2m$. Therefore,
\begin{equation*}
2\sqrt{l}-m=\frac{u}{\sqrt{m^2+u}+m}<1
\end{equation*}
and
\begin{equation*}
[2\sqrt{l}]=m, \quad \{2\sqrt{l}\} = \frac{u}{\sqrt{m^2+u}+m}.
\end{equation*}
Consequently,
\begin{equation*}
\frac{\{2\sqrt{l}\}}{4\sqrt{l}} = 
\frac{u}{m+\sqrt{m^2+u}}\frac{1}{2\sqrt{m^2+u}}\geq\frac{1}{4m^2}.
\end{equation*}
This inequality can be applied to estimate
\begin{equation*}
(4k-1)\sqrt{\xi} = (4k-1)\arccos{\frac{n}{2\sqrt{l}}}, \quad 
1 \leq n \leq [2\sqrt{l}].
\end{equation*}
Changing the variable $n$ to $[2\sqrt{l}]-n$ we have
\begin{equation*}
(4k-1)\sqrt{\xi} = (4k-1)\arccos{\frac{[2\sqrt{l}]-n}{2\sqrt{l}}}, \quad 
0 \leq n \leq [2\sqrt{l}]-1.
\end{equation*}
Further, trigonometric identities give
\begin{equation*} 
\begin{split} 
(4k-1)\sqrt{\xi} &= 2(4k-1)\arcsin{\sqrt{\frac{n+\{2\sqrt{l}\}}{4\sqrt{l}}}} \\ 
& \geq 2(4k-1)\arcsin\sqrt{\frac{\{2\sqrt{l}\}}{4\sqrt{l}}}
\gg \frac{k}{m}\gg \frac{k}{\sqrt{l}}\gg k^{1/2-\epsilon},
\end{split} 
\end{equation*}
as required.

Next, we insert the asymptotic expansion \eqref{eq:Phi2k} into $ET_2$.
The contribution of the error term in \eqref{eq:Phi2k} is majorized by
\begin{multline*}
ET_{2,1} := 
\sum_{k} h\left(\frac{4k}{K}\right) \frac{1}{k} 
\sum_{\substack{l\leq k^{1+\epsilon} \\ l\neq \Box}}\frac{1}{l^{3/4}}
\sum_{n<2\sqrt{l}}(4l-n^2)^{\theta}\frac{\xi^{1/4}}{(\sin\sqrt{\xi})^{1/2}} \times \\ 
\left| Y_0((4k-1)\sqrt{\xi})\right| 
\ll
\sum_{k} h\left(\frac{4k}{K}\right) \frac{1}{k} \sum_{l\leq k^{1+\epsilon}} 
\frac{1}{l^{3/4}} \sum_{n<2\sqrt{l}} \frac{(4l-n^2)^{\theta}}{k^{1/2} (\sin 
\sqrt{\xi})^{1/2}}.
\end{multline*}
According to \eqref{eq:xi} we have
\begin{equation*}
\sin\sqrt{\xi}=\frac{\sqrt{4l-n^2}}{2\sqrt{l}},
\end{equation*}
and therefore,
\begin{equation*}
ET_{2,1} \ll 
\sum_{k} h\left(\frac{4k}{K}\right)\frac{1}{k^{3/2}} \sum_{l \leq k^{1+\epsilon}} 
\frac{1}{l^{3/4}} \sum_{n<2\sqrt{l}}(4l-n^2)^{\theta-1/4}l^{1/4} 
\ll K^{\theta+1/4}.
\end{equation*}
It remains to estimate the contribution of the main term in \eqref{eq:Phi2k} 
given by
\begin{multline*}
ET_{2,2} := 
\sum_{k} h\left(\frac{4k}{K}\right) \sum_{l} \frac{V_k(l,0)}{l^{3/4}} 
\sum_{n<2\sqrt{l}} \mathscr{L}_{n^2-4l}(1/2) \times \\ 
\frac{\sin((4k-1)\sqrt{\xi}-\pi/4)}{\sqrt{4k-1}(\sin\sqrt{\xi})^{1/2}} 
= \sum_{l\ll K^{1+\epsilon}} \frac{1}{l^{1/2}} \sum_{n<2\sqrt{l}} 
\frac{\mathscr{L}_{n^2-4l}(1/2)}{(4l-n^2)^{1/4}} \times \\
\sum_{k} h\left(\frac{4k}{K}\right) V_k(l,0) \frac{\sin((4k-1) 
\arccos\frac{n}{2\sqrt{l}}-\pi/4)}{\sqrt{4k-1}}.
\end{multline*}
The inner sum over $k$ can be evaluated similarly to \cite[Lemma~7.3]{BF}, 
as we show now.
Using the Poisson summation formula
\begin{equation*}
\sum_{k}h\left(\frac{4k}{K}\right)V_k(l,0)\frac{\sin((4k-1)\arccos\frac{n}{2\sqrt{l}}-\pi/4)}{\sqrt{4k-1}}=\sum_{m}I(m),
\end{equation*}
where
\begin{multline*}
I(m)=\int_{-\infty}^{\infty}h\left(\frac{4y}{K}\right)V_y(l,0)\frac{\sin((4y-1)\arccos\frac{n}{2\sqrt{l}}-\pi/4)}{\sqrt{4y-1}}\times \\ \exp(-my)dy\ll
K\left|\int_{-\infty}^{\infty} h(y)\exp(ig(y))V_{yK/4}(l,0)\frac{dy}{\sqrt{yK-1}}\right|
\end{multline*}
with
\begin{equation*}
g(y)=\frac{K}{4}y\left( -2\pi m \pm 4\arccos\frac{n}{2\sqrt{l}}\right).
\end{equation*}
Integrating $a$ times by parts we obtain
\begin{equation*}
I(m)\ll \begin{cases}
\sqrt{K}/(Km)^a &  m\neq 0\\
\sqrt{K}/(K\arccos\frac{n}{2\sqrt{l}})^a & m=0.
\end{cases}
\end{equation*}
Finally, since
\begin{equation*}
\arccos\frac{n}{2\sqrt{l}}\gg \frac{1}{\sqrt{l}}
\end{equation*}
we have
\begin{equation*}
\sum_{m}I(m)\ll \frac{\sqrt{K}}{(K\arccos\frac{n}{2\sqrt{l}})^a}\ll \frac{\sqrt{K}}{(K/\sqrt{l})^a}
\end{equation*}
and
\begin{equation*}
ET_{2,2}\ll K^{-a}.
\end{equation*}

The last step is to evaluate the sum over $l=\Box$.
Making the change of variables from $l$ to $l^2$ we have to estimate
\begin{equation*}
E_{2,3}:=\sum_{k}h\left( \frac{4k}{K}\right)\sum_{l}\frac{V_k(l^2,0)}{l^{3/2}}\sum_{n<2l}\mathscr{L}_{n^2-4l^2}(1/2)\Phi_{2k}\left( \frac{n^2}{4l^2}\right).
\end{equation*}
Using the subconvexity estimate \eqref{eq:subcL} we have
\begin{equation*}
ET_{2,3}\ll \sum_{l\ll K^{1/2+\epsilon}}\frac{1}{l^{3/2}}\sum_{n<2l}(4l^2-n^2)^{\theta}\left|\sum_{k}h\left( \frac{4k}{K}\right)V_k(l^2,0)\Phi_{2k}\left( \frac{n^2}{4l^2}\right)\right|.
\end{equation*}
Changing the variable $n$ to $2l-n$ gives
\begin{equation*}
ET_{2,3}\ll \sum_{l\ll K^{1/2+\epsilon}}\sum_{n<2l}\frac{(nl)^{\theta}}{l^{3/2}}\left|\sum_{k}h\left( \frac{4k}{K}\right)V_k(l^2,0)\Phi_{2k}\left(\left(1-\frac{n^2}{4l^2}\right)^2\right)\right|.
\end{equation*}

Let
\begin{equation*}
\xi:=4(\arcsin\sqrt{\frac{n}{4l}})^2.
\end{equation*}
Consequently,
\begin{equation*}
\left( 1-\frac{n}{2l}\right)^2=\cos^2\sqrt{\xi}.
\end{equation*}
Applying \cite[Theorem~6.10]{BF1}, we have
\begin{equation*}
\Phi_{2k}(\cos^2\sqrt{\xi})=\frac{-\pi}{\xi^{1/4}(\sin{\sqrt{\xi}})^{1/2}}\left[(1+C_J)Z_J(\xi)+C_YZ_Y(\xi)\right].
\end{equation*}
Note that now
\begin{equation*}
(4k-1)\sqrt{\xi}=2(4k-1)\arcsin\sqrt{\frac{n}{4l}}\gg \frac{k}{\sqrt{l}}\gg k^{3/4}.
\end{equation*}
In the same way as in the case $l \neq \Box$ we prove
\begin{equation*}
ET_{2,3}\ll K^{-1/4+\theta}.
\end{equation*}
Finally,
\begin{equation*}
ET_2\ll ET_{2,1}+ET_{2,2}+ET_{2,3}\ll  K^{1/4+\theta}.
\end{equation*}
\end{proof}

\begin{rem} Using \cite[Thm.~6.5]{BF1} with sufficiently large $N$ depending on $A$ and taking more terms in the asymptotics for $C_Y, C_J$ we obtain equation \eqref{eq:Phi2k} with additional series of main terms plus the error term $$O\left( \frac{1}{k^{A+2}}\left| \sqrt{\xi}Y_0((4k-1)\sqrt{\xi})\right|\right).$$

As a result, following the proof of Lemma \ref{lem:secondtype} for any fixed $A>0$ we have
\begin{equation}\label{eq:secondtype}
\sum_{k}h\left( \frac{4k}{K}\right) ET_2\ll K^{-A}.
\end{equation}
\end{rem}
%%%%%%%%%%%%%%%%%%%%%%%%%%%%%%%%%%%%%%%%%%%%%%%%%%%%%%%%%%%%%%%%%%%%%%%%%%%%%%%%%%%%%%%%%%%%%%%%%%%%%%%%%%%%%%%%%%

%%%%%%%%%%%%%%%%%%%%%%%%%%%%%%%%%%%%%%%%%%%%%%%%%%%%%%%%%%%%%%%%%%%%%%%%%%%%%%%%%%%%%%%%%%%%%%%%%%%%%%%%%%%%%%%%%%

\section*{Acknowledgments}
The first author thanks Viktor A. Bykovskii for providing his handwritten notes with a full proof of Lemma \ref{fourexp} and for introducing her to the Rankin-Selberg method.

%%%%%%%%%%%%%%%%%%%%%%%%%%%%%%%%%%%%%%%%%%%%%%%%%%%%%%%%%%%%%%%%%%%%%%%%%%%%%%%%%%%%%%%%%%%%%%%%%%%%%%%%%%%%%%%%%%

\nocite{}

\end{document}